\documentclass[leqno,11pt]{scrartcl}
\usepackage[utf8]{inputenc}

\usepackage{amsmath,amssymb,amsthm}
\usepackage{amscd}
\usepackage{setspace} %%% Einstellung des Zeilenvorschubs mit \spacing{ } %%%
\usepackage{enumerate}
\usepackage{array}
\usepackage{tabularx}
\usepackage{multirow}
\usepackage{booktabs}

\theoremstyle{definition}

\newtheorem{remark}{Remark}

\theoremstyle{plain}
\newtheorem{theorem}{Theorem}
\newtheorem*{lemmaon}{Lemma}

\newtheorem{corollary}{Corollary}
\newtheorem*{proposition}{Proposition}
\theoremstyle{remark}

\newcommand{\C}{\mathbb{C}}

\newcommand{\N}{\mathbb{N}}

\newcommand{\Q}{\mathbb{Q}}
\newcommand{\R}{\mathbb{R}}
\newcommand{\Z}{\mathbb{Z}}
\newcommand{\trace}{\operatorname{trace}\,}

\newcommand{\oh}{\mathcal{\scriptstyle{O}}}

\let\leq\leqslant

\begin{document}

\begin{center}
\begin{huge}
\begin{spacing}{1.0}
\textbf{The maximal discrete extension of $SL_2(\oh_K)$ for an imaginary quadratic number field $K$}  
\end{spacing}
\end{huge}

\bigskip
by
\bigskip

\begin{large}
\textbf{Aloys Krieg, Joana Rodriguez, Annalena Wernz\footnote{Aloys Krieg, Joana Rodriguez, Annalena Wernz,\\ Lehrstuhl A für Mathematik, RWTH Aachen University, D-52056 Aachen \\ krieg@rwth-aachen.de, joana.rodriguez@rwth-aachen.de, annalena.wernz@rwth-aachen.de}}
\end{large}
\vspace{0.5cm}\\
%25 April 2018
\vspace{1cm}
\end{center}
\begin{abstract}
Let $\oh_K$ be the ring of integers of an imaginary quadratic number field $K$. 
In this paper we give a new description of the maximal discrete extension of the group $SL_2(\oh_K)$ inside $SL_2(\C)$, which uses generalized Atkin-Lehner involutions. Moreover we find a natural characterization of this group in $SO(1,3)$.
\end{abstract}

\newpage
%\section{The maximal discrete extension of $SL_2(\oh_K)$}

Throughout this paper let 
\[
 K = \Q(\sqrt{-m})\subset \C, \;m\in\N \;\text{squarefree},
\]
be an imaginary quadratic number field. Its discriminant and ring of integers are 
\[
 d_K = \begin{cases}
        -m \\ -4m
       \end{cases}
\quad \text{and}\quad \oh_K = \begin{cases}
                               \Z+\Z(1+\sqrt{-m})/2, & \text{if}\;\,m\equiv 3\!\!\!\pmod{4},  \\
                               \Z+\Z\sqrt{-m}, & \text{if}\;\, m\equiv 1,2 \!\!\!\pmod{4}.
                              \end{cases}
\]
Denote by
\[
 \Gamma_K := SL_2(\oh_K)
\]
the group of integral $2\times 2$ matrices of determinant $1$ and by 
\[
 \Gamma^*_K \leq SL_2(\C)
\]
its maximal discrete extension in $SL_2(\C)$ according to \cite{EGM}, chap. 7.4. If $\langle \alpha_1,\ldots,\alpha_n\rangle$ denotes the $\oh_K$-module generated by $\alpha_1,\ldots,\alpha_n$ and $\sqrt{\cdot}$ stands for an arbitrary (fixed) branch of the square root, we obtain a description of $\Gamma^*_K$ from \cite{EGM}, chap. 7.4.

\begin{proposition}\label{Propos_1} %%% Lemma 1
For the imaginary quadratic number field $K$, one has 
\[
 \Gamma^*_K = \left\{\frac{1}{\sqrt{\alpha\delta-\beta \gamma}} \begin{pmatrix}
                                                                    \alpha & \beta \\ \gamma & \delta
                                                                   \end{pmatrix};
\;\alpha,\beta,\gamma,\delta \in\oh_K,\,\langle \alpha \delta-\beta \gamma\rangle = \langle \alpha,\beta, \gamma,\delta\rangle \neq \{0\}\right\}
\]
satisfying
\[
 [\Gamma^*_K:\Gamma_K] = 2^{\nu}, \; \nu = \sharp\{p;\;p\;\text{prime},\; p\mid d_K\}.
\]
\end{proposition}

%\section{Description via generalized Atkin-Lehner involutions}
The main aim of this note ist to give an alternative description of $\Gamma^*_K$ in terms of generalized Atkin-Lehner involutions, which are more familiar in the theory of modular forms, as well as a characterization of $\Gamma^*_K$ inside the orthogonal group $SO(1,3)$.

Let $\omega:=m+\sqrt{-m}$. Then we see that $\gcd(d,\omega\overline{\omega}/d)=1$ for each squarefree divisor $d$ of $d_K$. Hence there are $u,v\in\Z$ such that
\begin{gather*}\tag{1}\label{gl_1}
 ud-v\omega\overline{\omega}/d=1, \quad \text{i.e.}\quad V_d:=\frac{1}{\sqrt{d}} \begin{pmatrix}
                                                                                 ud & v\omega \\ \overline{\omega} & d
                                                                                \end{pmatrix}
\in SL_2(\C).
\end{gather*}
A straightforward verification shows that the coset
\begin{gather*}\tag{2}\label{gl_2}
 \Gamma_K V_d = V_d \Gamma_K \subseteq \frac{1}{\sqrt{d}}\, \oh^{2\times 2}_K \cap SL_2(\C)
\end{gather*}
is independent of the particular choice of $u$ and $v$, hence well-defined by $d$. Moreover one has for squarefree divisors $d$ and $e$ of $d_K$
\begin{align*}
  V^2_d & \in \Gamma_K, \tag{3}\label{gl_3} \\[1ex]
% V_d \cdot V_e & \in \Gamma_K V_{de}, \tag{4}\label{gl_4}
\end{align*}
%Thus we have for arbitrary squarefree divisors $d,e$ of $d_K$ 
\begin{gather*}\tag{4}\label{gl_4}
 V_d\cdot V_e \in \Gamma_K V_f,\quad f= de/\gcd(d,e)^2.
\end{gather*}
Hence $V_d$ can be viewed as a generalization of the Atkin-Lehner involution (cf. \cite{AL}, sect. 2).

\begin{theorem}\label{Theorem_1}%%% Theorem 1
For the imaginary quadratic number field $K$, the group $\Gamma^*_K$ admits the description 
\[
 \Gamma^*_K = \bigcup_{d\mid d_K,\,d\,\text{squarefree}} \Gamma_K V_d.
\]
\end{theorem}

\begin{proof}
Denote the right-hand side by $G$. Given $M,N\in\Gamma_K$ and squarefree divisors $d,e$ of $d_K$ we obtain
\[
 MV_d\cdot NV_e \in \Gamma_K(V_d V_e) = \Gamma_K V_f
\]
due to \eqref{gl_2} and \eqref{gl_4}. Moreover \eqref{gl_3} and \eqref{gl_2} imply
\[
 (M V_d)^{-1} = V^{-1}_d M^{-1} \in V_d \Gamma_K = \Gamma_K V_d.
\]
Therefore $G$ is a group, which contains $\Gamma_K = \Gamma_K V_1$. As $\Gamma_K$ is a discrete subgroup of $SL_2(\C)$, this is also true for $G$. Thus we have $G\subseteq \Gamma^*_K$ as well as
\[
 [G:\Gamma_K] = [\Gamma^*_K:\Gamma_K] = 2^{\nu}
\]
due to the Proposition, hence $G = \Gamma^*_K$ follows.
\end{proof}

As this is the main result we sketch a direct proof. Start with a matrix $M$ in a discrete subgroup $G$ of $SL_2(\C)$ containing $\Gamma_K$. We conclude that $M\in r\oh^{2\times 2}_K$ for some $r\in\R$, $r>0$, just as in \cite{EGM}, chap. 7, Proposition 4.3. In view of $\det M = 1$ we may choose $r=1/\sqrt{d}$ for some $d\in\N$. Moreover the minimal possible $d$ turns out to be a squarefree divisor of $d_K$. Considering $MV^{-1}_d$ we end up with the case that $G$ is a subgroup of $SL_2(K)$, which was covered in \cite{EGM}, chap.7, Lemma 4.4.

\begin{remark}\label{Remark_1}%%% Remark 1
a) From Theorem \ref{Theorem_1} it is clear that $\Gamma_K$ is normal in $\Gamma^*_K$ and that the factor group $\Gamma^*_K/\Gamma_K$ is isomorphic to $C^{\nu}_2$. \\
b) If $d\,|\,d_K$, $d>1$ is squarefree, then the coefficients of the matrices in $\Gamma_K V_d$ in \eqref{gl_2} are algebraic integers of the biquadratic number field $\Q(\sqrt{d},\sqrt{-m})$. \\
c) There is no maximal discrete extension of $\Gamma_K$ inside $GL_2(\C)$ as the series of groups
\[
 \{e^{2\pi ij/n}M;\; j=0,\ldots,n-1,\,M\in\Gamma_K\}, \;\; n\in\N,
\]
shows.
\end{remark}

%\section{The embedding into $SO(1,3)$ over a field}

We want to give a natural characterization of the groups above in $SO(1,3)$. Therefore we consider the $4$-dimensional $\R$-vector space
\[
 V:=\{H\in \C^{2\times 2};\;H= \overline{H}^{tr}\}
\]
of Hermitian $2\times 2$ matrices over $\C$. Then
\[
 q:V\to \R, \quad H\mapsto \det H,
\]
is a quadratic form on $V$ of signature $(1,3)$. Let $SO(V,q)$ denote the attached special orthogonal group. Moreover let $SO_0(V,q)$ stand for the connected component of the identity element. It can be characterized by the fact that it maps the cone of positive definite matrices in $V$ onto itself. More precisely one has
\[
 SO_0(V,q)=\{\varphi\in SO(V,q); \;\trace \varphi(E)> 0\},\quad E = \begin{pmatrix}
                                 1 & 0 \\ 0 & 1
                                \end{pmatrix}.
\]
From \cite{Hn}, \S4, sect. 14, we quote the

\begin{lemmaon}\label{Lemma}%%% Lemma
The mapping
\begin{align*}
 & \phi: SL_2(\C) \to SO_0(V,q),\;\; M\mapsto \phi_M,  \\
 & \phi_M: V\to V, \;\; H\mapsto MH\overline{M}^ {tr},
\end{align*}
is a surjective homomorphism of groups with kernel $\{\pm E\}$.
\end{lemmaon}

Now we consider the situation over the rational numbers $\Q$. Let $V_K:= V\cap K^ {2\times 2}$ denote the associated $\Q$-vector space of dimension $4$ and 
\begin{align*}
 \widehat{\Sigma}_K: & = \bigl\{\varphi\in SO_0(V,q);\;\varphi(V_K) = V_K\bigr\},  \\
 \widehat{\Gamma}_K: & = \bigl\{\tfrac{1}{\sqrt{d}}M\in SL_2(\C);\;d\in \N,\,M\in\oh^{2\times 2}_K\bigr\}\supseteq SL_2(K).
\end{align*}

\begin{corollary}\label{Corollary_1}%%% Corollary 1
For the imaginary quadratic number field $K$, one has 
\[
 \phi(\widehat{\Gamma}_K) = \widehat{\Sigma}_K.
\]
\end{corollary}

\begin{proof}
``$\subseteq$'' This is obvious due to the Lemma.  \\
``$\supseteq$'' Start with $M= \left(\begin{smallmatrix}
                                 \alpha & \beta \\ \gamma & \delta
                                \end{smallmatrix}\right)
\in SL_2(\C)$ such that $\phi_M \in \widehat{\Sigma}_K$ in accordance with the Lemma. Assume $\alpha\neq 0$, because we may replace $M$ by $MJ$, $J =
\left(\begin{smallmatrix}
0 & -1 \\ 1 & 0
\end{smallmatrix}\right)$. Calculating $\phi_M(H)$ for a suitable basis of $V_K$ leads to 
\[
 \alpha\overline{\alpha} \in \Q,\quad \alpha^2, \beta/\alpha,\gamma/\alpha \in K.
\]
Thus we obtain $\alpha \in \tfrac{1}{\sqrt{d}} \oh_K$ for a suitable $d\in\N$ and
\[
 M= \begin{pmatrix}
     1 & 0 \\ \gamma/\alpha & 1
    \end{pmatrix}
\begin{pmatrix}
     \alpha & 0 \\ 0 & 1/\alpha
    \end{pmatrix}
\begin{pmatrix}
     1 & \beta/\alpha\\ 0 & 1
    \end{pmatrix} \in \widehat{\Gamma}_K.  
\]
\vspace{-6ex}\\
\end{proof}
\vspace{2ex}

%\section{The characterization of $SL_2(\oh_K)$ and its maximal discrete extension in $SO(V,q)$}

Now consider the lattice $\Lambda_K:= V_K\cap\oh^{2\times 2}_K$ of integral Hermitian matrices and the group of lattice automorphisms in $SO_0(V,q)$ 
\[
 \Sigma^*_K:=\bigl\{\varphi \in SO_0(V,q);\;\varphi(\Lambda_K) = \Lambda_K\bigr\}\leq \widehat{\Sigma}_K.
\]

Thus we get a very natural description of $\Gamma^*_K$ in this context.

\begin{theorem}\label{Theorem_2}%%% Theorem 2
For the imaginary quadratic number field $K$, one has
\[
 \phi(\Gamma^*_K) = \Sigma^*_K.
\]
\end{theorem}

\begin{proof}
``$\subseteq$'' This is a consequence of the Lemma and Theorem \ref{Theorem_1}, because $\phi_{V_d}(\Lambda_K)$ equals $\Lambda_K$ for all squarefree divisors $d$ of $d_K$.  \\
``$\supseteq$'' $\phi^{-1}(\Sigma^*_K)$ is a discrete subgroup of $SL_2(\C)$ containing $\Gamma_K$ due to the Lemma. As $\Gamma^*_K$ is maximal discrete, we get
\[
\phi^{-1}(\Sigma^*_K) \subseteq \Gamma^*_K.
\]
\vspace*{-7ex}\\
\end{proof}
\vspace{1ex}

Now it is easy to characterize $\Gamma_K$ as the so-called discriminant kernel. For this purpose consider the dual lattice
\begin{align*}
 \Lambda^{\sharp}_K: & = \bigl\{S\in V_K;\;\trace (SH)\in\Z\;\text{for all}\; H\in \Lambda_K\bigr\}  \\
 & = \biggl\{\begin{pmatrix}
                  s_1 & s \\ \overline{s} & s_2
                 \end{pmatrix};\; s_1,s_2\in\Z,\, s\in \tfrac{1}{\sqrt{d_K}}\oh_K \biggr\} \supseteq \Lambda_K.
\end{align*}
Any $\varphi\in \Sigma^*_K$ satisfies $\varphi(\Lambda^{\sharp}_K) = \Lambda^{\sharp}_K$. Hence we may define
\[
 \Sigma_K : = \bigl\{\varphi\in \Sigma^*_K;\; \varphi\;\text{induces id on}\;\Lambda^{\sharp}_K/\Lambda_K\bigr\}.
\]

\begin{corollary}\label{Corollary_2}%%% Corollary 2
For the imaginary quadratic number field $K$, one has
\[
 \phi(\Gamma_K) = \Sigma_K.
\]
\end{corollary}

\begin{proof}
For $M\in\Gamma_K$ it is easy to verify that $\phi_M\in \Sigma_K$. If $d\,|\,d_K$ is squarefree, a simple calculation yields
\[
 \phi_{V_d} \in \Sigma_K \Leftrightarrow d=1.
\]
Thus $\phi_M\in \Sigma_K$, $M\in \Gamma^*_K$, holds if and only if $M\in \Gamma_K$, due to the Lemma in combination with Theorem \ref{Theorem_1}.
\end{proof}
\vspace{1ex}

\begin{remark}\label{Remark_2}
a) Similar results hold for the group
\[
 \Gamma_0(N):=\biggl\{\left(\begin{smallmatrix}
                            \alpha & \beta \\ \gamma & \delta
                           \end{smallmatrix}\right)
\in SL_2(\Z);\; \gamma \equiv 0 \!\!\!\pmod{N}\biggr\}
\]
instead of $\Gamma_K$ (cf. \cite{K1}, sect. 6) and moreover in the case of the paramodular group of degree $2$ (cf. \cite{GKr}, \cite{Koe}).  \\
b) Considering the Hermitian modular group of degree $2$ over $K$ (cf. \cite{B}) the group $\Gamma^*_K$ also appears, if one wants to compute its maximal normal extension (cf. \cite{We}).
\end{remark}

% \nocite{H}  %%% Ausgabe der Literatur ohne Zitat im Text!!!

%============================================

\bibliographystyle{plain}
\renewcommand{\refname}{Bibliography}
\bibliography{bibliography_krieg} 

\end{document}